\documentclass[reqno,11pt]{amsart}
\usepackage[utf8]{inputenc}
\usepackage{ upgreek }
\usepackage{float,epsfig,here}
\usepackage{enumerate}
\usepackage{amsmath,amssymb}
\newtheorem{remark}{Remark}

\newtheorem{assumption}{Assumption}

\newtheorem{theorem}{Theorem}
\newtheorem{corollary}{Corollary}
\allowdisplaybreaks
\newtheorem{proof*}{Proof*}
\newtheorem{definition}{Definition}

\begin{document}

\baselineskip=17pt

\title[]
{Risk-sensitive Semi-Markov Decision Problems with Discounted Cost and General Utilities}
\thanks{}
\author[]{Arnab Bhabak}
\address{Department of Mathematics\\
Indian Institute of Technology Guwahati\\
Guwahati, Assam, India}
\email{bhabak@iitg.ac.in}

\author[]{Subhamay Saha}
\address{Department of Mathematics\\
Indian Institute of Technology Guwahati\\
Guwahati, Assam, India}
\email{saha.subhamay@iitg.ac.in}


\date{}

\begin{abstract} In this article we consider risk-sensitive control of semi-Markov processes with a discrete state space. We consider general utility functions and discounted cost in the optimization criteria. We consider random finite horizon and infinite horizon problems. Using a state augmentation technique we characterise the value functions and also prescribe optimal controls.
\vspace{2mm}


\vspace{2mm}


\end{abstract}

\maketitle

\section{Introduction} In this article we consider risk-sensitive optimization problems for semi-Markov processes on a countable state space. We consider non-negative running cost function, which is also assumed to be bounded above. The cost functional also includes the discount factor. The aim of the decision maker is to minimize the expected utility of the discounted cost accumulated over an infinite time horizon. In order to solve the infinite horizon problem we first consider an auxiliary random finite horizon problem, where the optimization is upto the Nth jump time of the process. The finite horizon problem can also be of independent interest. Then using the finite horizon problem we solve the infinite horizon problem via an appropriate limiting argument. 

Since in an optimal control problem, the cost payable is a random quantity, the standard approach is to minimize the expectation of the cost functional. Such an approach is referred to as the risk-neutral approach, because expectation minimization does not take into account the ``risk" factor which is generally captured by the higher moments. This can be a major issue in real life scenarios. In order to address this issue, the popular remedy is to consider expectation of the exponential of the cost functional. This is known as the risk-sensitive approach. Since the pioneering work of Howard and Matheson \cite{Matheson72}, there has been a lot of work on risk-sensitive control of both discrete and continuous time stochastic processes. Risk-sensitive control problems has been studied in literature for discrete time Markov chains, see \cite{Rieder14, Meyn02, Sobel87, Stettner07, Marcus96, Jask07}, for diffusions, see \cite{Biswas18, Kumar10, McEneaney95, Robin05, Nagai96, Whittle90}, for continuous-time Markov chains, see \cite{Saha14, Wei19, Zhang19, Pal13, Chen19}. The literature on risk-sensitive control of semi-Markov processes is very few. In \cite{Cadena16} the authors consider risk-sensitive control of semi-Markov processes with average cost criterion. There the state space is assumed to be finite and the sojourn time distributions are assumed to have a compact support. In \cite{Lian18}, the authors consider risk-sensitive control of semi-Markov processes on a fixed finite horizon $[0,T]$. To the best of our knowledge, this is first paper dealing with risk-sensitive control of semi-Markov process for a discounted cost criterion over the infinite horizon. In this paper we consider general utility functions $U$ which can be either concave or convex. The classical risk-sensitive case corresponds to the exponential utility function. But here we allow power utilities, logarithmic utilities and so on. When $U$ is convex the decision maker is said to be risk-averse, while he/she is said to be risk-seeking when the utility function is concave. For further discussions and interpretations of utility functions see \cite{Rieder14}. The rest of the paper is organized as follows. In Section 2, we describe our model and the control problem. In Section 3 we investigate the random finite horizon problem. Finally in Section 4, we analyse the infinite horizon problem.
\section{The Control Model}
The semi-Markov decision problem(SMDP) model that we are interested in is
\begin{align*}
\{E,A,(A(i), i\in E), Q(.,.\vert i,a), C(i,a),U(.)\},
\end{align*}
where the individual components has the following interpretation:
\begin{itemize}
	\item $E$ is a countable state space. Without loss of generality, we take $E=\{1,2,\ldots\}$.
	\item $A$ is the action space, which is assumed to be Borel space endowed with the Borel $\sigma$-algebra $\mathcal{A}$.
	\item $A(i)\in \mathcal{A}$ denotes the set of all admissible actions in state $i$.
	 Let $K:=\{(i,a)\vert i\in E, a\in A(i)\}$ be the set of all admissible state action pairs. 
	 \item $Q(\cdot,\cdot\vert i,a)$ is a semi-Markov kernel on $[0,\infty)\times E$ given $K$. We assume that $Q(0,j|i,a)=0$ for any $j\in E$ and $(i,a)\in K$. It describes the transition mechanism of the controlled process. Thus if $a\in A(i)$ is the action chosen in state $i$, then for any $t>0$ and $j\in E$, $Q(t,j|i,a)$ is the joint probability that the sojourn time in state $i$ will be less than or equal to $t$ and the next transition will be into state $j$. 
	 \item $C:K\rightarrow [0,\bar{c}]$ is a measurable running cost function with $0<\bar{c}<\infty$.
	 \item $U:[0,\infty)\to \mathbb{R}$ denotes a utility function, which is assumed to be continuous and strictly increasing. 
\end{itemize}

Now we describe the evolution of the controlled semi-Markov process.  At time $0$, which is the initial decision epoch, the system is in state $i_{0}$. Depending upon the state of the system the controller chooses an action $a_0\in A(i_{0})$. As a consequence of this choice of action the system remains at $i_{0}$ until time $t_1$. At time $t_{1}$ the system jumps to the next state $i_{1}$ according to the transition law $Q(dt_1,i_1\vert i_0,a_0)$. A discounted cost equal to $\int_{0}^{t_1}e^{-\alpha u}C(i_0,a_0)du$ is generated. Now in state $i_{1}$, depending on the current state, previous state, sojourn time, and previously selected action the controller chooses an action $a_{1}\in A(i_{1})$ and the same sequence of events repeat. Based on this evolution and we obtain a history $h_{n}=(i_{0},a_{0},t_1,i_{1},a_{1},t_{2},.....,i_{n-1},a_{n-1},t_{n},i_{n})$ up to the nth jump of the described process. Here $t_{k}$ denotes the $kth$ jump time with the assumption $t_{0}=0$, $i_k$ is the state after the $kth$ jump and $a_k$ is the action chosen at the $kth$ jump time. Let $H_n$ denote the set of all possible histories upto the $nth$ jump time. $H_n$ is endowed with a Borel $\sigma-$field.\\
Next, we describe the policies which govern the choice of action by the decision-maker.
\begin{definition}
	A history dependent policy $\pi:=\{\pi_{n},n\geq 0\}$ is sequence of measurable functions $\pi_n:H_n\to A$ such that $\pi_n(h_n)\in A(i_n)$. A history dependent policy $\pi$ is said to be Markov if there exists a sequence $\{f_n\}$ of measurable functions $f_n:[0,\infty)\times E\to A$, such that $f_n(t,i)\in A(i)$ and $\pi_n(h_n)=f_n(t_n,i_n)$. In this case we write $\pi=\{f_n\}_{n\geq 0}$. If $f_n=f$ for some common function $f$ for all $n$, then the Markov policy is said to be stationary. We will sometimes denote a stationary policy by the common function $f$. We denote by $\Pi$, $\Pi^{M}$, $\Pi^{S}$ the set of all history dependent, Markov and stationary policies respectively.
\end{definition}
For each $i\in E$ and $\pi\in \Pi$ by the well-known Tulcea's Theorem( \cite{Lasserre96}, Proposition C.10), there exist a unique probability space $(\Omega,\mathcal{F},\mathbb{P}^{\pi}_{i})$ and stochastic processes $\{T_{n},X_{n},A_{n}\}_{n\geq 0}$ such that, for each $t\in [0,\infty)$, $j\in E$, $C\in \mathcal{A}$ and $n\geq 0$,
\begin{align*}
&\mathbb{P}^{\pi}_{i}(T_{0}=0,X_{0}=i)= 1,\\
&\mathbb{P}^{\pi}_{i}(A_{n}\in C\vert h_{n})= \delta_{\pi_{n}(h_n)}(C),\\
&\mathbb{P}^{\pi}_{i}(T_{n+1}-T_{n}\leq t, X_{n+1}=j\vert h_{n},a_{n},b_{n})= Q(t,j\vert x_{n},a_{n},b_{n}),
\end{align*} 
where $T_{n}$, $X_{n}$ and $A_{n}$ denote the $nth$ jump time, the state and the action chosen by the decision maker at the nth jump time and $\delta$ denotes the Dirac measure. The expectation operator with respect to $\mathbb{P}^{\pi}_{i}$ is denoted by $\mathbb{E}^{\pi}_{i}$. In order to avoid the possibility of an infinite 
number of jumps within a finite time interval we make the following standard assumption.
\begin{assumption}
	 There exist constants $\delta> 0$ and $\epsilon> 0$ such that
		\begin{align}\label{finite jump}
		\sup_{(i,a)\in K} Q(\delta,E\vert  i, a) \leq  1 - \epsilon.
		\end{align}
\end{assumption}	

\begin{remark}
Assumption 1 means that the sojourn time at any state and under any action exceeds $\delta$ with a probability at least $\epsilon$. If $T_{\infty}=\lim_{n\to \infty}T_n$, then it is well known that (see Proposition 2.1 of \cite{Song11}), if \eqref{finite jump} holds then $\mathbb{P}^\pi_i(T_{\infty}=\infty)=1$ for any $i\in E$ and $\pi\in \Pi$.
\end{remark} 
We also define the continuous time processes $\{X(t), A(t), t\in [0,\infty)\}$ by
$$X(t)=X_{n},\quad A(t)=A_{n}, \quad \mbox{for}\quad T_{n}\leq t < T_{n+1},\,\,t\in [0,\infty)\,\, \mbox{and}\,\,n\geq 0\,.$$
Now we describe the cost criteria. Let $\alpha>0$ be a discount factor. We consider SMDPs over both finite(random) and infinite horizons. For $N\geq 1$, the total discounted cost accumulated upto the $Nth$ jump time is given by
\begin{align*}
C_{N}=\int_{0}^{T_N}e^{-\alpha u}C(X_{u},A_{u})du=\sum_{n=1}^N e^{-\alpha T_{n-1}}\int_{0}^{T_n-T_{n-1}}e^{-\alpha t}C(X_n,A_n)dt,   
\end{align*}
and the total discounted cost accumulated over the infinite time horizon is given by
\begin{align*}
C_{\infty}=\int_{0}^{\infty}e^{-\alpha u}C(X_{u},A_{u})du=\sum_{n=1}^{\infty} e^{-\alpha T_{n-1}}\int_{0}^{T_n-T_{n-1}}e^{-\alpha t}C(X_n,A_n)dt.   
\end{align*}
Instead of the standard expectation minimization, here we consider the following minimization problems:
\begin{align*}
&\inf_{\pi\in \Pi} \mathbb{E}^{\pi}_{i}\big[U(C_{N})\big], \hspace{.2cm}i\in E \hspace{.2cm} \mbox{and}\\
&\inf_{\pi\in \Pi}\mathbb{E}^{\pi}_{i}\big[U(C_{\alpha}^{\infty})\big], \hspace{.2cm}i\in E.
\end{align*}
For $i\in E$ and $\pi \in \Pi$, let $$J_N^\pi(i)=\mathbb{E}^{\pi}_{i}\big[U(C_{N})\big]\quad \mbox{and}\quad J_{\infty}^\pi(i)=\mathbb{E}^{\pi}_{i}\big[U(C_{\infty})\big].$$ Also let $$J_N(i)=\inf_{\pi\in \Pi} J_N^\pi(i)\quad \mbox{and}\quad J_{\infty}(i)=\inf_{\pi\in \Pi}J_{\infty}^\pi(i).$$ A policy $\pi^* \in \Pi$ is said to be optimal for the finite horizon problem if $J_N^{\pi^*}(i)=J_N(i)$ for all $i\in E$. Similarly, a policy $\pi^*$ is said to be optimal for the infinite horizon problem if $J_{\infty}^{\pi^*}(i)=J_{\infty}(i)$ for all $i$ in $E$. We wish to characterise $J_N(\cdot)$ and $J_{\infty}(\cdot)$ and find optimal policies for both finite and infinite horizon problems.

\section{Finite Horizon Problem}
We first consider the optimization problem upto the $Nth$ jump time. We will use a state augmentation technique to convert the original problem to a standard risk-neutral problem. Similar state augmentation technique has been used in the context of discrete time MDP in \cite{Rieder14} and in the context of SMDP in \cite{Lian18}. We augment the state process to include the accumulated discounted cost. More precisely, we consider the augmented controlled state process $\{T_n,X_n,C_n,n\geq 0\}$ where $T_n$ and $X_n$ are as before and $C_n$ is the accumulated discounted cost upto the $nth$ jump time. For $i\in E$, $(t,\lambda)\in [0,\infty)\times [0,\infty)$ and $a\in A(i)$, the controlled transition law of the augmented state process is given by 
$$\hat{Q}(B\times \{j\}\times C|t,i,\lambda,a)=\int_{0}^{\infty}1_{B}(t+s)\delta_{\lambda+\frac{C(i,a)}{\alpha}(e^{-\alpha t}-e^{-\alpha(t+s)})}(C)Q(ds,j|i,a),$$ where $1$ is the indicator function, $B$ and $C$ are Borel subsets of $[0,\infty)$, $j\in E$. In this augmented set-up we redefine the various policy sets. But for economy of notation, we use the same notations for the augmented policy sets. Thus, in particular, in the augmented set-up a Markov policy is given by $\pi=\{f_n\}$ where $f_n:[0,\infty)\times E\times [0,\infty)\to A$ are measurable functions such that $f_n(t,i,\lambda)\in A(i)$ for all $(t,i,\lambda)$. 

Suppose $\mathbb{E}_{(t,i,\lambda)}$ is the expectation operator corresponding to the initial condition\\ $T_0=t, X_0=i, C_0=\lambda$. Then for $n=0,1,\ldots,N$ and $\pi \in \Pi$ define the value functions,

\begin{align}\label{augmentation}
V_{n \pi}(t,i,\lambda)&=\mathbb{E}^{\pi}_{(t,i,\lambda)}\big[U\big(e^{-\alpha t}\int_{0}^{T_n}e^{-\alpha u}C(X_{u},A_{u})du+\lambda\big)\big]=\mathbb{E}^{\pi}_{(t,i,\lambda)}[U(C_n)],\nonumber\\ & (t,i,\lambda)\in [0,\infty)\times E\times [0,\infty),\nonumber\\
 V_{n}(t,i,\lambda)&=\inf_{\pi\in \Pi}V_{n \pi}(t,i,\lambda),\quad (t,i,\lambda)\in [0,\infty)\times E\times [0,\infty). 
\end{align}
Thus $J_N(i)=V_N(0,i,0)$. The state augmentation now allows us to think of the optimization problem as a finite horizon discrete time Markov decision process with state process $\{T_n,X_n,C_n,n\geq 0\}$, zero one stage cost and terminal cost function $g(t,i,\lambda)=U(\lambda)$. Now define the set
\begin{align}\label{set}
B([0,\infty)\times E\times [0,\infty))= \{&v:[0,\infty)\times E\times [0,\infty)\rightarrow [0,\infty)\,\,\mbox{is measurable and}\nonumber\\& U(\lambda)\leq v(t,i,\lambda)\leq U(e^{-\alpha t}\frac{\bar{c}}{\alpha}+\lambda)\,\,\forall (t,i,\lambda)\}.
\end{align} Let $F$ denote the set of all measurable functions $f:[0,\infty)\times E\times [0,\infty)\to A$ such that $f_n(t,i,\lambda)\in A(i)$ for all $(t,i,\lambda)$. Then for $v\in B([0,\infty)\times E\times [0,\infty))$ and $f\in \mathbb{F}$, we define the operators
\begin{align}\label{operator1}
T_{f}v(t,i,\lambda):=\sum_{j}\int v\big(t+s,j,\lambda+\frac{C(i,f(t,i,\lambda))}{\alpha}(e^{-\alpha t}-e^{-\alpha(t+s)})\big)Q(ds,j\vert i,f(t,i,\lambda))
\end{align}
and
\begin{align}\label{minimizing operator}
(Tv)(t,i,\lambda)
:=\inf_{a\in A(i)}\sum_{j}\int v\big(t+s,j,\lambda+\frac{C(i,a)}{\alpha}(e^{-\alpha t}-e^{-\alpha(t+s)})\big)Q(ds,j\vert i,a).
\end{align}
We say that $f\in F$ is a minimizer of $v$ if $Tv=T_{f}v$. Observe that the operators $T_{f}v$ and $T$ are both monotone. We need to impose the following assumption.  
\begin{assumption} For each $(t,i,\lambda) \in [0,\infty)\times E\times [0,\infty)$,\\
	(i) $A(i)$ is compact.\\
	(ii) $c(i,a)$ is continuous on $A(i)$.\\
	(iii) $\sum_{j}\int v\big(t+s,j,\lambda+\frac{C(i,a)}{\alpha}(e^{-\alpha t}-e^{-\alpha(t+s)})\big)Q(ds,j\vert i,a)$ is continuous on $A(i)$, for each $v\in B([0,\infty)\times E\times [0,\infty))$.
\end{assumption}
We have the following:
\begin{theorem}\label{finite}
Suppose $Assumptions$ 1 and 2 hold.	For $n=1,....,N$, we have the following.
	\begin{itemize}
		\item[(a)] For any policy $\pi=(f_{0},f_{1},....)\in \Pi^{M}$, we have the iteration: $V_{n \pi}=T_{f_{0}} T_{f_{1}}....T_{f_{n-1}}U$, where $T_{f_i}$ is as in \eqref{operator1}.
		\item[(b)] $V_{0}(t,i,\lambda)=U(\lambda)$ and $V_{n}=TV_{n-1}$ i.e.,\\
		\begin{align}\label{Bellman eqn}
		V_{n}(i,\lambda,t)=\inf_{a\in A(i)} \sum_{j}\int V_{n-1}\big(t+s,j,\lambda+\frac{C(i,a)}{\alpha}(e^{-\alpha t}-e^{-\alpha(t+s)})\big)Q(ds,j\vert i,a).
		\end{align}
		Also, $V_{n}\in B([0,\infty)\times E\times [0,\infty))$ for all $n=0,1,\ldots, N$.
		\item[(c)] For every $n=1,2,...,N$ there exists a minimizer $f^{*}_{n}\in F$ of $V_{n-1}$ and $(f_N^*,f_{N-1}^*,\ldots,f_1^*)$ is an optimal Markov policy for the finite horizon optimization problem.		
	\end{itemize}
\end{theorem}
\begin{proof} We prove (a) by induction on $n$.\\
Firstly, $V_{0 \pi}(t,i,\lambda)=U(\lambda)$. Now for $n=1$ we have,
\begin{align*}
V_{1 \pi}(t,i,\lambda)=&\mathbb{E}^{\pi}_{(t,i,\lambda)}\big[\big(\lambda+e^{-\alpha t}\int_{0}^{T_1}e^{-\alpha u}C(i,f_{0}(t,i,\lambda))du\big)\big]\\
=&(T_{f_{0}}U)(t,i,\lambda).
\end{align*}
Now suppose that the statement holds for $V_{n-1  \pi}$. So we consider $V_{n \pi}$:

\begin{align*}
(T_{f_{0}}....T_{f_{n-1}}U)(t,i,\lambda)&=T_{f_{0}}(T_{f_{1}}....T_{f_{n-1}}U)(t,i,\lambda)\\&=\sum_{j}\int V_{n-1 \bar{\pi}}\big(t+s,j,\lambda+\frac{C(i,f_{0}(t,i,\lambda))}{\alpha}(e^{-\alpha t}-e^{-\alpha(t+s)})\big)Q(ds,j\vert i,f_0(t,i,\lambda))\\
&=\sum_{j}\int \mathbb{E}^{\bar{\pi}}_{(t+s,j,\lambda^{\prime})}\big[U\big(\lambda+\frac{C(i,f_{0}(t,i,\lambda))}{\alpha}(e^{-\alpha t}-e^{-\alpha(t+s)})\\&+e^{-\alpha (t+s)}\int_{0}^{T_{n-1}}e^{-\alpha u}C(X_{u},a_{u})du\big)\big]Q(ds,j\vert i,f_0(t,i,\lambda))\\
&=\mathbb{E}^{\pi}_{(t,i,\lambda)}\big[U\big(e^{-\alpha t}\int_{0}^{T_{n}}e^{-\alpha u}C(X_{u},A_{u})du+\lambda \big)\big
]\\
&=V_{n \pi}(i,\lambda,t),
\end{align*}
where $\bar{\pi}$ is the one shifted policy and $\lambda^{\prime}=\lambda+\frac{C(i,f_{0}(t,i,\lambda))}{\alpha}(e^{-\alpha t}-e^{-\alpha(t+s)})$. Hence we got our desired result for $V_{n \pi}$. Then by induction argument we showed (a).\\

(b) and (c): The proof of $(b)$ and $(c)$ follows by Assumption 2 and standard theory of discrete time MDP, see Chapter 3 of \cite{Lasserre96}.
\end{proof}
\begin{corollary}\label{cor1}
	In the case of $U(\lambda)=(\frac{1}{\gamma})e^{\gamma \lambda}$ with $\gamma \neq 0$, we have $V_{n}(i,\lambda,t)=e^{\gamma \lambda}h_{n}(t,i)$ and $J_{N}(i)=h_{N}(0,i)$. And $h_{n}$ satisfies the iteration given by $h_{0}(i,\lambda)=\frac{1}{\gamma}$	and 
	\begin{align*}
	h_{n}(t,i)=\inf_{a\in A(i)}\sum_{j}\int e^{\gamma \frac{C(i,a)}{\alpha}(e^{-\alpha t}-e^{-\alpha (t+s)})} h_{n-1}(t+s,j)Q(ds,j\vert i,a).
	\end{align*}
\end{corollary}
\begin{proof} We will prove by induction on $n$. For $n=0$ we have $V_{0}(t,i,\lambda)=(\frac{1}{\gamma})e^{\gamma \lambda}=e^{\gamma \lambda}(\frac{1}{\gamma})=e^{\gamma \lambda}h_{0}$. Hence the statement is true for $n=0$. Now suppose it is true for $n-1$. From \eqref{Bellman eqn} we get,
\begin{align*}
V_{n}(t,i,\lambda)&=\inf_{a\in A(i)}\sum_{j}\int V_{n-1}\big(t+s,j,\lambda+\frac{C(i,a)}{\alpha}(e^{-\alpha t}-e^{-\alpha(t+s)})\big)Q(ds,j\vert i,a)\\
&= \inf_{a\in A(i)}\sum_{j}\int e^{\gamma (\lambda+\frac{C(i,a)}{\alpha}(e^{-\alpha t}-e^{-\alpha(t+s)}))}h_{n-1}(t+s,j)Q(ds,j\vert i,a)\\
&=e^{\gamma \lambda}\inf_{a\in A(i)}\sum_{j}\int e^{\gamma \frac{C(i,a)}{\alpha}(e^{-\alpha t}-e^{-\alpha (t+s)})} h_{n-1}(t+s,j)Q(ds,j\vert i,a)
\end{align*}
Hence, the statement follows by setting $$h_{n}(t,i)=\inf_{a\in A(i)}\sum_{j}\int e^{\gamma \frac{C(i,a)}{\alpha}(e^{-\alpha t}-e^{-\alpha (t+s)})} h_{n-1}(t+s,j)Q(ds,j\vert i,a).$$
\end{proof}

\section{Infinite Horizon Problem}

Now to consider the infinite horizon problem. Like in the finite horizon case we again consider the augmented set-up and define the following value functions.
{\small\begin{align*}\label{augmentation in infinite case}
&V_{\infty \pi}(t,i,\lambda):=\mathbb{E}^{\pi}_{(t,i,\lambda)}\big[U\big(e^{-\alpha t}\int_{0}^{\infty}e^{-\alpha u}C(X_{u},A_{u})du+\lambda\big)\big],\,\,\pi\in \Pi,\,(t,i,\lambda)\in [0,\infty)\times E\times [0,\infty).
\end{align*}}
\begin{align*}
V_{\infty}(t,i,\lambda)=\inf_{\pi\in \Pi}V_{\infty \pi}(t,i,\lambda),\,\,(t,i,\lambda)\in [0,\infty)\times E\times [0,\infty).
\end{align*}
For a stationary policy $\pi=\{f\}$, we will will write $V_{\infty \pi}$ as $V_f$. It is easy to see that, $J_{\infty}(i)=V_{\infty}(0,i,0)$ for all $i\in E$. For the infinite horizon problem, we will deal with convex and concave utility functions separately. First we analyse the concave case.
\subsection{Concave Utility Function.} let $U:\mathbb{R_{+}} \rightarrow \mathbb{R}$ be a concave utility function. We introduce one more notation $\bar{v}(t,\lambda)= U(e^{-\alpha t}\frac{\bar{c}}{\alpha}+\lambda)$. It is straight forward to see that $U(\lambda) \leq V_{\infty}(i,\lambda, t)\leq  \bar{v}(t,\lambda)$. Thus $V_{\infty}\in B([0,\infty)\times E\times [0,\infty))$ where $B([0,\infty)\times E\times [0,\infty))$ is given by \eqref{set}.

\begin{theorem}\label{concave} Suppose that Assumptions 1 and 2 hold. Also suppose that the derivative $U^{\prime}(0)$ exists. Then the following statements hold.
	\begin{itemize}
		\item[(a)] $V_{\infty}$ is the unique solution of $v=Tv$ in $B([0,\infty)\times E\times [0,\infty))$ for $T$ defined in (\ref{minimizing operator}).
		Moreover, $T^{n}U \uparrow V_{\infty}$ and $T^{n}\bar{v} \downarrow V_{\infty}$ as $n\rightarrow \infty$.
		\item[(b)]  There exists a minimizer $f^{*}\in F$ of $V_{\infty}$ and the stationary policy determined by $f^*$
		is an optimal policy for the infinite horizon problem.
	\end{itemize}
\end{theorem}
\begin{proof} (a) Here we first show that $V_{n}=T^{n}U \uparrow V_{\infty}$ as $n\rightarrow \infty$. It is known that for $U:\mathbb{R_{+}}\rightarrow \mathbb{R}$ increasing and concave we have the inequality
$$U(\lambda_{1}+\lambda_{2})\leq U(\lambda_{1})+U^{\prime}_{-}(\lambda_{1})\lambda_{2}, \quad \lambda_{1},\lambda_{2}\geq 0,$$
where $U^{\prime}_{-}$ is the left-hand side derivative of $U$ that exists since U is concave. Moreover, $U^{\prime}_{-}(\lambda)\geq 0$ and is decreasing. For $(t,i,\lambda)\in [0,\infty)\times E\times [0,\infty)$ and $\pi \in \Pi$ we have,
\begin{align*}
V_{n}(t,i,\lambda)&\leq V_{n \pi}(t,i,\lambda)\leq V_{\infty \pi}(t,i,\lambda)=\mathbb{E}^{\pi}_{(t,i,\lambda)}\big[U\big(e^{-\alpha t}\int_{0}^{\infty}e^{-\alpha u}C(X_{u},A_{u})du)+\lambda\big)\big]\\
&=\mathbb{E}^{\pi}_{(t,i,\lambda)}\big[U\big(e^{-\alpha t}\int_{0}^{T_{n}}e^{-\alpha u}C(X_{u},A_{u})du+\lambda+e^{-\alpha t}\int_{T_{n}}^{\infty}e^{-\alpha u}C(X_{u},A_{u})du\big)\big]\\
&\leq \mathbb{E}^{\pi}_{(t,i,\lambda)}\big[U\big(e^{-\alpha t}\int_{0}^{T_{n}}e^{-\alpha u}C(X_{u},A_{u})du)+\lambda\big)\big]\\&+\mathbb{E}^{\pi}_{(t,i,\lambda)}\big[U^{\prime}_{-}\big(e^{-\alpha t}\int_{0}^{T_{n}}e^{\alpha u}C(X_{u},A_{u})du)+\lambda\big)
\times \big(e^{-\alpha t}\int_{T_{n}}^{\infty}e^{-\alpha u}C(X_{u},A_{u})du\big)\big]\\
&\leq V_{n \pi}(t,i,\lambda)+U^{\prime}_{-}(\lambda)e^{-\alpha t}\frac{\bar{c}}{\alpha}\mathbb{E}^{\pi}_{t,i,\lambda}(e^{-\alpha T_{n}}).
\end{align*}
Now using \eqref{finite jump}, it can be shown by induction that $$\mathbb{E}^{\pi}_{(t,i,\lambda)}[e^{-\alpha T_{n}}]\leq (1-\epsilon+\epsilon e^{-\alpha \delta})^{n},$$ for all $n$. Thus we have,
\begin{align*}
V_{n}(t,i,\lambda)&\leq V_{\infty \pi}(t,i,\lambda)\leq V_{n \pi}(i,\lambda,t)+\epsilon_{n}(t,\lambda),
\end{align*} where $\epsilon_{n}(\lambda, t)=U^{\prime}_{-}(\lambda)e^{-\alpha t}\frac{\bar{c}}{\alpha}(1-\epsilon+\epsilon e^{-\alpha \delta})^{n}$. Thus $\lim_{n\rightarrow \infty}\epsilon_{n}(\lambda,t)=0$. Taking infimum over all policies we get,
$$V_{n}(t,i,\lambda)\leq V_{\infty}(i,\lambda,t)\leq V_{n}(t,i,\lambda)+\epsilon_{n}(t,\lambda).$$  
Now as $n\rightarrow \infty$ we have $V_{n}=T^{n}U\uparrow V_{\infty}$ for $n\rightarrow \infty$. Now, we try to show that $V_{\infty}=TV_{\infty}$. Note that $V_{n}\leq V_{\infty}$ for all $n$. Using the fact that $T$ is increasing we have $V_{n+1}=TV_{n}\leq TV_{\infty}$ for all n. Letting $n\rightarrow \infty$ implies $V_{\infty}\leq TV_{\infty}$.\\
For the reverse inequality we have from above $V_{n}+\epsilon_{n}\geq V_{\infty}$. Applying the $T$ operator and also using its monotonicity we get,
\begin{align*}
T(V_{n}+\epsilon_{n})(t,i,\lambda)&=\inf_{a\in A(i)}\sum_{j}\int \biggl(V_{n}\big(t+s,j,\lambda+\frac{C(i,a)}{\alpha}(e^{-\alpha t}-e^{-\alpha (t+s)})\big)\\&+\epsilon_{n}\big(t+s,\lambda+\frac{C(i,a)}{\alpha}(e^{-\alpha t}-e^{-\alpha (t+s)})\big)\biggr)Q(ds,j\vert i,a)\\
&\leq V_{n+1}+\epsilon_{n+1}.
\end{align*}          
Hence, we have $V_{n+1}+\epsilon_{n+1}\geq T(V_{n}+\epsilon_{n})\geq TV_{\infty}$. Now letting $n\rightarrow \infty$ we obtain $V_{\infty}\geq TV_{\infty}$. So together we have $V_{\infty}=TV_{\infty}$. Now,
\begin{align*}
T\bar{v}(t,\lambda)&=\inf_{a\in A(i)}\sum_{j}\int U\big(e^{-\alpha (t+s)}\frac{\bar{c}}{\alpha}+\lambda+\frac{C(i,a)}{\alpha}(e^{-\alpha t}-e^{\alpha (t+s)})\big)Q(ds,j\vert i,a)\\
&\leq \inf_{a\in A(i)}\sum_{j}\int U\big(e^{-\alpha (t+s)}\frac{\bar{c}}{\alpha}+\lambda+\frac{\bar{c}}{\alpha}(e^{-\alpha t}-e^{\alpha (t+s)})\big)Q(ds,j\vert i,a)=\bar{v}(t,\lambda).
\end{align*}
Hence we have $T^{n}\bar{v}$ is a decreasing sequence. Moreover, we have by iteration
\begin{align*}
&(T^{n}U)(t,i,\lambda)=\inf_{\pi\in \Pi^{M}}\mathbb{E}^{\pi}_{(t,i,\lambda)}\big[U\big(e^{-\alpha t}\int_{0}^{T_{n}}e^{-\alpha u}C(X_{u},A_{u})du+\lambda\big)\big]\\
&(T^{n}\bar{v})(t,i,\lambda)=\inf_{\pi\in \Pi^M}\mathbb{E}^{\pi}_{(t,i,\lambda)}\big[U\big(e^{-\alpha t}\frac{\bar{c}}{\alpha}e^{-\alpha T_{n}}+e^{-\alpha t}\int_{0}^{T_{n}}e^{-\alpha u}C(X_{u},A_{u})du+\lambda\big)\big].
\end{align*}         
Again using the inequality $U(\lambda_{1}+\lambda_{2})- U(\lambda_{1})\leq U^{\prime}_{-}(\lambda_{1})\lambda_{2}$, we have,       
\begin{align*}
0&\leq (T^{n}\bar{v})(i,\lambda,t)- (T^{n}U)(i,\lambda,t)\\
&\leq \sup_{\pi\in \Pi}\mathbb{E}^{\pi}_{(t,i,\lambda)}\big[U\big(e^{-\alpha t}\frac{\bar{c}}{\alpha}e^{-\alpha T_{n}}+e^{-\alpha t}\int_{0}^{T_{n}}e^{-\alpha u}C(X_{u},A_{u})du+\lambda)\\&-U(e^{-\alpha t}\int_{0}^{T_{n}}e^{-\alpha u}C(X_{u},A_{u})du+\lambda\big)\big]\\&
\leq \epsilon_{n}(\lambda,t),
\end{align*}          
and $\epsilon_{n}(\lambda,t)$ converges to zero as $n\rightarrow \infty$. Hence $T^{n}\bar{v}\downarrow V_{\infty}$ as $n\rightarrow \infty$.

For the uniqueness let $w$ be another solution of $w=Tw$ with $U(\lambda)\leq w \leq \bar{v}$. Then, by iteration we get $T^{n}U\leq w \leq T^{n}\bar{v}$, for all $n$. So, by taking $n\rightarrow \infty$ in the inequality and using the fact that $T^{n}\bar{v}\downarrow V_{\infty}$ and $T^{n}U\uparrow V_{\infty}$ we get the uniqueness.

 (b)  The existence of a minimizer follows from our Assumptions and standard measurable selection theorem. Now using the fact that $V_{\infty}(t,i,\lambda)\geq U(\lambda)$ we obtain
 $$V_{\infty}=\lim\limits_{n\rightarrow \infty}T^{n}_{f^{*}}V_{\infty}\geq \lim\limits_{n\rightarrow \infty}T^{n}_{f^{*}}U=\lim\limits_{n\rightarrow \infty}V_{n(f^{*},f^{*,.....})}=V_{f^{*}}\geq V_{\infty},$$
 where the last equation follows using dominated convergence theorem. Hence we get the optimality of the stationary policy given by $f^*$.\end{proof}
 Obviously it can be shown that for a policy $\pi=(f_{0},f_{1},.....)\in \Pi^{M}$ we have the following cost iteration: $V_{\infty \pi}=\lim_{n\rightarrow \infty}(T_{f_{0}}T_{f_{1}}...T_{f_{n-1}})U$ . For a stationary policy $\pi=(f,f,.....)$ the cost iteration becomes $V_{f}=T_fV_{f}$.     
          
The above Theorem tells us that from a computational point of view, the value function of the infinite horizon optimization problem can be approximated arbitrarily close by sandwiching between 
$T^{n}U$ and $T^{n}\bar{v}$. Moreover, also the policy improvement algorithm works in this setting. For that,
for $v\in B([0,\infty)\times E \times [0,\infty))$, we define the operator $(Lv)(t,i,\lambda,a):=\sum_{j}\int v\big(t+s,j,\lambda+\frac{C(i,a)}{\alpha}(e^{-\alpha t}-e^{-\alpha(t+s)})\big)Q(ds,j\vert i,a)$. Also, for any $f\in F$ and $(t,i,\lambda)\in [0,\infty)\times E \times [0,\infty)$ we set $D(t,i,\lambda,f):=\{a\in A(i): (Lv_{f})(i,\lambda,t,a)< V_{f}(i,\lambda,t)\}$. Then by arguments analogous to Theorem 4 in \cite{Rieder14}, the following Theorem can be proved.
\begin{theorem}[Policy Improvement]
	 Suppose $f\in F$. \\
	(a) Define $h\in F$ by $h(\cdot)\in D(\cdot,f)$ if the set $D(.,f)$ is non-empty and otherwise $h=f$. Then $V_h \leq V_f$ and the improvement is strict in states where $D(\cdot ,f) \neq \phi$.\\
	(b) If $D(\cdot,f)= \phi$ for all states, then $V_{f}=V_{\infty}$ and $f$ defines an optimal policy.\\
	(c) Suppose $f_{k+1}$ is a minimizer of $V_{f_{k}}$
	for $k\geq 0$ where $f_{0}=f$. Then $V_{f_{k+1}}\leq V_{f_{k}}$
	and $\lim_{k\rightarrow \infty}V_{f_{k}}=V_{\infty}$.
\end{theorem}
           
\subsection{Convex Utility Function}
Now we look into the case of a convex utility function. For that $U:\mathbb{R_{+}} \rightarrow \mathbb{R}$ be a convex utility function.  The functions $V_{n \pi}$, $V_{n}$, $V_{\infty \pi}$, $V_{\infty}$ are defined as in the previous section.
\begin{theorem}\label{convex}
Under Assumptions 1 and 2, the conclusions of Theorem \ref{concave} hold for convex utility function as well.
\end{theorem}
\begin{proof}
 The proof of this Theorem is similar to that of Theorem \ref{concave}. The main difference is that now we need to use the following property of convex function. Since $U:\mathbb{R_{+}}\rightarrow \mathbb{R}$  strictly increasing and convex we have the inequality 
$$U(\lambda_{1}+\lambda_{2})\leq U(\lambda_{1})+U^{\prime}_{+}(\lambda_{1}+\lambda_{2})\lambda_{2},    \quad\lambda_{1},\lambda_{2}\geq 0,$$
where $U^{\prime}_{+}$ is the right-hand side derivative of $U$ that exists since $U$ is convex. Moreover, $U^{\prime}_{+}(\lambda)\geq 0$ and $U^{\prime}_{+}$ is increasing. For $(t,i,\lambda)\in [0,\infty)\times E\times [0,\infty)$ and $\pi \in \Pi$ we have,
\begin{align*}
V_{n}(t,i,\lambda)&\leq V_{n \pi}(t,i,\lambda)\leq V_{\infty \pi}(t,i,\lambda)=\mathbb{E}^{\pi}_{(t,i,\lambda)}\big[U\big(e^{-\alpha t}\int_{0}^{\infty}e^{-\alpha u}C(X_{u},A_{u})du)+\lambda\big)\big]\\
&=\mathbb{E}^{\pi}_{(t,i,\lambda)}\big[U\big(e^{-\alpha t}\int_{0}^{T_{n}}e^{-\alpha u}C(X_{u},A_{u})du)+\lambda+e^{-\alpha t}\int_{T_{n}}^{\infty}e^{-\alpha u}C(X_{u},A_{u})du\big)\big]\\
&\leq \mathbb{E}^{\pi}_{(t,i,\lambda)}\big[U\big(e^{-\alpha t}\int_{0}^{T_{n}}e^{-\alpha u}C(X_{u},A_{u})du)+\lambda\big)\big]\\&+\mathbb{E}^{\pi}_{(t,i,\lambda)}\big[U^{\prime}_{+}\big(e^{-\alpha t}\int_{0}^{\infty}e^{-\alpha u}C(X_{u},A_{u})du)+\lambda\big)
\times \big (e^{-\alpha t}\int_{T_{n}}^{\infty}e^{-\alpha u}C(X_{u},A_{u})du\big)\big]\\
&\leq V_{n \pi}(t,i,\lambda)+U^{\prime}_{+}(e^{-\alpha t}\frac{\bar{c}}{\alpha}+\lambda)(e^{-\alpha t}\frac{\bar{c}}{\alpha})(1-\epsilon+\epsilon e^{-\alpha \delta})^{n},\\
&=V_{n \pi}(t,i,\lambda)+\delta_{n}(t,\lambda),
\end{align*}
where $\delta_{n}(\lambda,t)=U^{\prime}_{+}(e^{-\alpha t}\frac{\bar{c}}{\alpha}+\lambda)(e^{-\alpha t}\frac{\bar{c}}{\alpha})(1-\epsilon+\epsilon e^{-\alpha \delta})^{n}$. As $n\rightarrow \infty$, $\lim\limits_{n\rightarrow \infty}\delta_{n}(\lambda,t)=0$. Taking the infimum over all policies in the
above inequality yields
$$V_{n}(t,i,\lambda)\leq V_{\infty}(t,i,\lambda)\leq V_{n}(t,i,\lambda)+\delta_{n}(t,\lambda).$$
Letting $n\rightarrow \infty$ yields $\lim\limits_{n\rightarrow \infty}T^{n}U=V_{\infty}.$
Again, using the same inequality we have
\begin{align*}
0&\leq (T^{n}\bar{v})(t,i,\lambda)- (T^{n}U)(t,i,\lambda)\\
&\leq \sup_{\pi\in \Pi}\mathbb{E}^{\pi}_{(t,i,\lambda)}\big[U\big(e^{-\alpha t}\frac{\bar{c}}{\alpha}e^{-\alpha T_{n}}+e^{-\alpha t}\int_{0}^{T_{n}}e^{-\alpha u}C(X_{u},A_{u})du+\lambda)\\&-U(e^{-\alpha t}\int_{0}^{T_{n}}e^{-\alpha u}C(X_{u},A_{u})du+\lambda \big)\big]\\&
\leq U^{\prime}_{+}(e^{-\alpha t}\frac{\bar{c}}{\alpha}+\lambda)(e^{-\alpha t}\frac{\bar{c}}{\alpha})(1-\epsilon+\epsilon e^{-\alpha \delta})^{n} =\delta_{n}(t,\lambda).
\end{align*}    
The rest of the arguments are same as Theorem \ref{concave}.
\end{proof}
The policy improvement algorithm for convex utility functions works in exactly the same way as for the concave case.
The following corollary is easy to deduce from Theorems \ref{concave} and \ref{convex}.

\begin{corollary}\label{cor2}
	In case $U(\lambda)=(\frac{1}{\gamma})e^{\gamma \lambda}$ with $\gamma \neq 0$, we obtain $V_{\infty}(t,i,\lambda)=e^{\gamma \lambda}h_{\infty}(t,i)$ and $J_{\infty}(i)=h_{\infty}(0,i)$. And the function $h_{\infty}$ is the unique fixed point of
	\begin{align}
	h_{\infty}(i,t)=\inf_{a\in A(i)}\sum_{j}\int e^{\gamma \frac{C(i,a)}{\alpha}(e^{-\alpha t}-e^{-\alpha (t+s)})} h_{\infty}(t+s,j)Q(ds,j\vert i,a).
	\end{align}
	with $\frac{1}{\gamma}\leq h_{\infty}(t,i)\leq \frac{1}{\gamma}e^{\gamma e^{-\alpha t}\frac{\bar{c}}{\alpha}}$.
\end{corollary}
\begin{remark}Few remarks are in order.
\begin{enumerate}
\item We see from Corollaries \ref{cor1} and \ref{cor2} that in the case of exponential utility, the case which is classically referred to as the risk-sensitive control in literature, the value functions split. Thus, for the exponential utility, the optimal controls as given by Theorems \ref{finite}, \ref{concave} and \ref{convex} will not depend on the accumulated cost and thus they belong to the original non-augmented policy sets.
\item The dependence of the optimal policies on the jump times is not surprising. Because, in the risk-sensitive control literature it is known that in presence on discounting, for discrete-time Markov chains, continuous-time Markov chains as well as for diffusions,  optimal policies do depend on time.
                                
\end{enumerate}
\end{remark}

\bibliographystyle{plain}
\bibliography{bib}

\end{document}